\documentclass[12pt,reqno]{amsart}
\usepackage{cases}
\usepackage{amscd}
\usepackage{amsfonts}
\usepackage{amssymb}
\usepackage{amsmath}
\usepackage{graphicx}
\usepackage{epstopdf}
\usepackage{subfigure}
\usepackage{amsthm}
\usepackage{stmaryrd}

\theoremstyle{remark}
\newtheorem{example}{\textbf{Example}}[section]
\numberwithin{equation}{section}

\usepackage{color}
\usepackage{datetime}

\makeatletter
\newcommand\figcaption{\def\@captype{figure}\caption}
\newcommand\tabcaption{\def\@captype{table}\caption}
\makeatother

\usepackage{geometry}
\usepackage{algorithm}
\usepackage{multirow}

\def\bq{\begin{equation}}
\def\eq{\end{equation}}
\def\bqs{\begin{equation*}}
\def\eqs{\end{equation*}}
\def\bsqs{\begin{subequations}}
\def\esqs{\end{subequations}}
\def\ba{\begin{aligned}}
\def\ea{\end{aligned}}
\def\br{\begin{eqnarray}}
\def\er{\end{eqnarray}}
\def\brr{\bq\begin{array}{rlll}}
\def\err{\end{array}\eq}

\def\text#1{\hbox{#1}}

\newtheorem{alg}{Algorithm}[section]
\newtheorem{thm}{Theorem}[section]
\newtheorem{lem}{Lemma}[section]

\newtheorem{rem}{Remark}[section]

\newcommand{\bsub}{\begin{subequations}}
\newcommand{\esub}{\end{subequations}$\!$}

\title[mixed DG without interior penalty]{
On the SAV-DG method for a class of fourth order gradient flows}

\author[H.~Liu, P. ~Yin]{Hailiang Liu$^\dagger$ and Peimeng Yin$^\S$}
\address{$^\ddagger$ Iowa State University, Department of Mathematics, Ames, IA 50011} \email{hliu@iastate.edu}
\address{$^\S$ Wayne State University, Department of Mathematics, Detroit, MI 48202} \email{pyin@wayne.edu}
\keywords{Gradient flows, energy dissipation, DG method, SAV approach.}
\subjclass{65N12, 65N30,  35K35}

\date{\today}

\begin{document}

\begin{abstract}
For a class of fourth order gradient flow problems, integration of the scalar auxiliary variable (SAV) time discretization with the penalty-free discontinuous Galerkin (DG) spatial discretization leads to SAV-DG schemes.
These schemes are linear and shown unconditionally energy stable. But the reduced linear systems are rather expensive to solve due to the dense coefficient matrices. In this paper,  we provide a procedure to pre-evaluate the auxiliary variable in the piecewise polynomial space. As a result the computational complexity of $O(\mathcal{N}^2)$ reduces to $O(\mathcal{N})$ when exploiting the conjugate gradient (CG) solver. This hybrid SAV-DG method is more efficient and able to deliver satisfactory results of high accuracy.  This was also compared with solving the full augmented system of the SAV-DG schemes.  
\end{abstract}

\maketitle

\bigskip


\section{Introduction} This paper is concerned with efficient numerical approximations to  a class of fourth order gradient flows \cite{FK99}:
\begin{align}\label{fourthPDE}
u_t  =  - \left(\Delta+\frac{a}{2}\right)^2 u -\Phi'(u) , \; x\in \Omega \subset \mathbb{R}^d, \; t>0, 
\end{align}
which governs the evolution of a scalar time-dependent unknown $u=u(x, t)$ in a convex bounded domain $\Omega\subset \mathbb{R}^d$, $\Phi$ is a nonlinear function and $a$ serves as a physical parameter. The model equation (\ref{fourthPDE}) describes important physical processes in nature. Typical application examples include the Swift-Hohenberg (SH) equation \cite{SH77} and the extended Fisher--Kolmogorov equation \cite{DS88, PT95}.

It is known that under appropriate boundary conditions,  equation (\ref{fourthPDE}) features a decaying free energy
\bq\label{engdisO}
\frac{d}{dt}\mathcal{E}(u) =  -\int_{\Omega}|u_t|^2dx \leq 0,
\eq
where
\bq\label{eqpdf}
\mathcal{E}(u) = \int_\Omega \frac{1}{2}\left(\mathcal{L} u \right)^2 + \Phi(u)dx, \quad \mathcal{L}=-\left(\Delta+\frac{a}{2}\right).
\eq
This energy dissipation law as a fundamental property of (\ref{fourthPDE}) is always desirable
for numerical approximations, and often crucial to eliminate numerical results that are not physical.

For the spatial discretization, we follow the penalty free discontinuous Galerkin (DG) method introduced in \cite{LY18}. The key idea is to introduce $q=\mathcal{L}u$ so that the resulting semi-discrete DG scheme
becomes
\begin{subequations}\label{SemiDG+}
\begin{align}
& (u_{ht}, \phi) = -A(q_h, \phi) - (\Phi'(u_h), \phi), \\
& (q_h, \psi)=A(u_h, \psi),
\end{align}
\end{subequations}
for all $\phi, \ \psi$ in the same DG space as for $u_h, q_h$.  Here $A(q_h, \cdot)$ is the DG discretization of $(\mathcal{L} q, \cdot)$. This spatial DG discretization avoids the use of penalty parameters (called penalty-free DG method) in the numerical flux on interior cell interfaces. It also inherits most of the advantages of the usual DG methods (see e.g. \cite{HW07, Ri08, Sh09}), such as high order accuracy, flexibility in hp-adaptation, capacity to handle domains with complex geometry.

In order to formulate an energy dissipative scheme with the time discretization, the linear terms  in (\ref{SemiDG+}) can be treated implicitly, but nonlinear terms have to be handled with care. The IEQ-DG method introduced  in \cite{LY19} is to integrate the DG method with the method of invariant energy quadratization (IEQ) \cite{Y16, ZWY17}. It boils down to solving an augmented system involving the dynamics of the auxiliary variable $U=\sqrt{\Phi(u_h)+B}$. We remark that the IEQ approach is remarkable as it allows one to construct linear, unconditionally energy stable schemes for a large class of gradient flows (see, e.g. \cite{Y16, YZW17, YZWS17, ZWY17, LY19, LY20}).
We refer the readers to \cite{LY19} for more references to earlier results on both the DG approximation and the time discretization.

As pointed out in \cite{LY19}, one could also integrate the same DG method with the so-called SAV approach  \cite{SXY18} by introducing an auxiliary variable
$
r=\sqrt{\int_\Omega\Phi(u(x, t))dx+B}.
$
This transforms (\ref{SemiDG+}) into another augmented system.
As for the IEQ-DG method, here one can also obtain a closed linear system for  $(u_h^{n+1}, q_h^{n+1})$ only.
Unfortunately, such systems involve dense coefficient matrices and rather expensive to solve.

There are two ways to get around this obstacle:  (i) find a path to lower the computational complexity of solving the reduced linear system; or  (ii) return to the full augmented system with $(u_h^{n+1}, q_h^{n+1}, r^{n+1})$ as unknowns.  For (i) we introduce a special procedure to pre-compute $r^{n+1}=r(t^{n+1})$ in the piecewise polynomial space based on a linear DG solver; with such obtained $r^{n+1}$, we solve the SAV-DG schemes with reduced computational cost. 
This treatment is interesting in its own sake. We name it the hybrid SAV-DG method. For (ii),  the full augmented system indeed involves only sparse coefficient matrices.  
Here the full system contains one more equation since $r$ does not depend on $x$.  In contrast,  the full system with $(u_h^{n+1}, q_h^{n+1}, U^{n+1})$ as unknowns for the IEQ-DG method contains $N(k+1)$ more equations. Here $N$ is the total number of the 1-D meshes, and $k$ the degree of DG polynomials. The advantage of the IEQ-DG method lies in the simplicity of its reduced system. 

Comparing the linear systems of the above three SAV-DG type-schemes, we see that the coefficient matrices are all symmetric,  but it is time-dependent and dense for the reduced system, time-dependent and sparse for the full augmented system, and time-independent and sparse for the hybrid SAV-DG. Indeed, our numerical tests confirm that the hybrid SAV-DG algorithm performs the best.

\subsection{Organization} This paper is organized as follows:  In Section 2, we formulate a unified semi-discrete DG method for the fourth order equation (\ref{fourthPDE}) subject to two different boundary conditions. In Section 3, we present SAV-DG schemes,  show the energy dissipation law, and discuss several ways to efficiently implement the schemes.  In Section 4, we provide a procedure to pre-evaluate the auxiliary variable and then present the according algorithms. In Section 5, we verify the good performance of the hybrid SAV-DG using  several numerical examples. Finally  some concluding remarks are given in Section 6. 

{\bf Notation:}   Throughout this paper, we use the notation $\Pi$ to indicate the usual piecewise $L^2$ projection in the sense of inner product with $\forall \phi \in V_h$,
$$
(\Pi w, \phi)=(w, \phi), \quad \forall \phi \in V_h,
$$
where $V_h$ is the discontinuous Galerkin finite element space.

\section{Spatial DG discretization } To introduce the hybrid SAV-DG  algorithm,  
we need to first recall some conventions about the semi-discrete DG discretization introduced in \cite{LY18}.
To be specific, we only consider homogeneous boundary conditions of form
\begin{equation}\label{BC}
\text{(i)} \ u \text{ is periodic};  \quad \text{or (ii)}   \ \partial_\mathbf{n} u=\partial_\mathbf{n} \Delta u=0, \quad x\in \partial \Omega,
\end{equation}
where $\mathbf{n}$ stands for the unit outward normal to the boundary $\partial \Omega$.

For the fourth order PDE  (\ref{fourthPDE}),  we set 
$q=\mathcal{L}u$ so that the model admits the following mixed form
\begin{equation}\label{mix}
\left \{
\begin{array}{rl}
    u_t = &- \mathcal{L} q-\Phi'(u),\\
    q = & \mathcal{L} u.
\end{array}
\right.
\end{equation}
Let the domain $\Omega$ be a union of  shape regular meshes $\mathcal{T}_h=\{K\}$, with the mesh size $h_K = \text{diam}\{K\}$ and $h=\max_{K} h_K$.  We denote the set of the interior interfaces by $\Gamma^0$,  the set of all boundary faces by $\Gamma^\partial$,  and the discontinuous Galerkin finite element space by
$$
V_h = \{v\in L^2(\Omega) \ : \ v|_{K} \in P^k(K), \ \forall K \in \mathcal{T}_h \},
$$
where $P^k(K)$ denotes the set of polynomials of degree no more than $k$ on element $K$. If the normal vector on the element interface $e\in \partial K_1 \cap \partial K_2$ is oriented from $K_1$ to $K_2$, then the average $\{\cdot\}$ and the jump $[\cdot]$ operator are defined by
$$
\{v\} = \frac{1}{2}(v|_{\partial K_1}+v|_{\partial K_2}), \quad [v]=v|_{\partial K_2}-v|_{\partial K_1}, 
$$
for any function $v \in V_h$,  where $v|_{\partial K_i} \ (i=1,2)$ is the trace of $v$ on $e$ evaluated from element $K_i$.
Then the DG method for (\ref{mix}) is to find $(u_h(\cdot, t), q_h(\cdot, t))\in V_h\times V_h  $ such that
\begin{subequations}\label{SemiDG}
\begin{align}
& (u_{ht}, \phi) = -A(q_h, \phi) - (\Phi'(u_h), \phi), \\
& (q_h, \psi)=A(u_h, \psi),
\end{align}
\end{subequations}
for all $\phi, \ \psi \in V_h$.  The initial data for $u_h$ is taken as the piecewise $L^2$ projection, denoted by  $u_h(x,0)=\Pi u_0(x)$.
In the above scheme formulation $A(q_h, \phi)$ is the DG discretization of $(\mathcal{L} q, \phi)$ and $A(u_h, \psi)$ is the DG discretizationvof  $(\mathcal{L} u, \psi)$.

The precise form of $A(\cdot, \cdot)$ depending on the types of boundary conditions is given as follows:
$$
A(w,v)=A^0(w,v)+A^b(w,v)
$$
with
\bq \label{A0}
A^0(w,v)= \sum_{K\in \mathcal{T}_h} \int_K \left( \nabla w \cdot \nabla v  - \frac{a}{2} w v \right)dx + \sum_{e\in \Gamma^0} \int_e \left( \{\partial_\nu w\}[v]+ [w]\{\partial_\nu v\} \right)ds.
\eq
Here $A^b(\cdot, \cdot)$ are given below for each respective type of boundary conditions:
\begin{subequations}\label{bd+}
\begin{align}
\text{for (i) of (\ref{BC}) } \qquad & A^b(w, v)=\frac{1}{2}  \int_{\Gamma^{\partial} } \left( \{\partial_\nu w\}[v]+ [w]\{\partial_\nu v\} \right)ds, \\
\text{for (ii) of (\ref{BC})} \qquad  &  A^b(w,v)=0.
\end{align}
\end{subequations}
Note that for periodic case in (\ref{bd+}a) the left boundary and the right boundary are considered as same, for which we use the factor $1/2$ to avoid recounting.

One can verify that the semi-discrete scheme (\ref{SemiDG}) satisfies a discrete energy dissipation law (see \cite{LY19})
 \bqs
\frac{d}{dt}\mathcal E(u_h,q_h) = -\int_{\Omega}|u_{ht}|^2dx \leq 0,
\eqs
where
\bq\label{oeq}
\mathcal E(u_h,q_h)=\int_{\Omega} \frac{1}{2}|q_h|^2 + \Phi(u_h) dx.
\eq
For non-homogeneous boundary conditions, it only requires a modification by adding some source terms in the DG formulation. Of course, the energy dissipation also needs to be refined to account for the boundary effects.

\section{Time discretization}\label{sec3}
With time discretization 
using the SAV approach (cf.  \cite{SXY18}), we introduce
 $$
 r=r(t):=\sqrt{\int_\Omega\Phi(u_h(x, t))dx+B}
 $$
 where  $B$ is so chosen that this quantity is well-defined, and consider  the following enlarged system:
find $(u_h(\cdot, t), q_h(\cdot, t)) \in V_h  \times V_h $ and $r=r(t)$ such that
\begin{subequations}\label{SemiDG++}
\begin{align}
(u_{ht}, \phi) = &- A(q_h,\phi) -r\left(b(u_h),\phi \right),\\
(q_h, \psi) = & A(u_h,\psi),\\
r_t = & \frac{1}{2}\int_\Omega b(u_h) u_{ht}dx,
\end{align}
\end{subequations}
for all $\phi, \psi \in V_h $. Here we use the notation
\begin{align}\label{hw}
 b(w(\cdot))= \frac{\Phi'(w(\cdot))}{\sqrt{\int_\Omega\Phi(w(x))dx+B} }.
\end{align}
The initial data for the above scheme is chosen as
$$
u_h(x, 0)=\Pi u_0(x), \quad r(0)=\sqrt{\int_\Omega\Phi(u_0(x))dx+B},
$$
where $\Pi$ denotes the piecewise $L^2$ projection into $V_h$.

One can verify that  a modified energy of form
\begin{align}\label{ee}
E(u_h, q_h, r)=\frac{1}{2}\int_{\Omega} q_h^2 dx+r^2 =\mathcal E(u_h, q_h)+B
\end{align}
satisfies the following dissipation inequality
\bqs
\frac{d}{dt}E(u_h, q_h, r) = -\int_{\Omega}|u_{ht}|^2dx \leq 0.
\eqs
Using the Euler-forward time discretization, we obtain the first order SAV-DG scheme:
find  $(u^{n}_h, q_h^{n}) \in V_h  \times V_h $ and $r^n=r(t^n)$ such that for any for $\phi, \psi \in V_h $,
\begin{subequations}\label{FPDGFull1st+}
\begin{align}
\left(  D_t u_h^n, \phi \right) = & - A(q_h^{n+1},\phi)-r^{n+1} \left( b(u_h^n), \phi \right), \\
(q_h^{n},\psi) = & A(u_h^{n}, \psi),\\
D_t r^n = & \frac{1}{2}\int_\Omega b(u_h^n) D_t u_h^n dx,
\end{align}
\end{subequations}
 The initial data
$
u_h^0=u_h(x, 0), \; r^0= r(0).
$
Here we used $D_tv^n=\frac{v^{n+1}-v^n}{\Delta t}$.

Reformulation (\ref{SemiDG++}) also allows for even higher order in time discretization.  To illustrate this we only consider a second order SAV-DG scheme: find $(u^{n}_h, q_h^{n}) \in V_h  \times V_h $ such that for for all $ \phi, \psi
\in V_h $,
\begin{subequations}\label{FPDGFull+}
	\begin{align}
	\left(  D_t u_h^n, \phi \right) = & - A(q_h^{n+1/2},\phi)-r^{n+1/2}\left(b(u^{n,*}_h),\phi \right),\\
	(q_h^{n}, \psi) = & A(u_h^{n},\psi),\\
    D_t r^n = & \frac{1}{2}\int_\Omega b(u^{n,*}_h) D_t u_h^ndx,
	\end{align}
\end{subequations}
where $v^{n+1/2}=(v^n+v^{n+1})/2$ for $v=u_h, q_h, r$, and $u^{n,*}_h$ is defined by
\begin{align}\label{u8}
u^{n, *}_h=& \frac{3}{2}u_h^n-\frac{1}{2}u_h^{n-1}.
\end{align}
Here instead of  $u_h^{n+1/2}$ we use $u^{n, *}_h$ to avoid the use of iteration steps in updating the numerical solution, while still maintaining second order accuracy in time. When $n=0$ in (\ref{u8}), we simply take  $u_h^{-1}=u_h^{0}$.

Both scheme (\ref{FPDGFull1st+}) and (\ref{FPDGFull+}) are unconditionally energy stable.
\begin{thm}\label{firstorder+}
(i) Scheme (\ref{FPDGFull1st+}) admits a unique solution $(u_h^{n}, q_h^{n})$, and for
 $E^n := E(u_h^n, q_h^n, r^n)$, we have
	\begin{align}\label{engdis1st+}
	E^{n+1} =  E^n - \frac{\| u_h^{n+1} - u_h^n\|^2}{\Delta t}-\frac{1}{2}\|q_h^{n+1} - q_h^n\|^2-|r^{n+1} - r^n|^2.
	\end{align}
for any $\Delta t>0$.	\\
(ii)  Scheme (\ref{FPDGFull+}) admits a unique solution,  and
\bq\label{engdis}
E^{n+1} =  E^n - \frac{\| u_h^{n+1} - u_h^n\|^2}{\Delta t}
\eq
for any $\Delta t >0$.
\end{thm}
The proof of this result is deferred to Appendix A.

Though SAV-DG schemes are linear and unconditionally energy stable, their numerical implementations cannot be handled as for the IEQ-DG schemes in \cite{LY19}.   To see this, we follow \cite{LY19} to rewrite (\ref{FPDGFull1st+}) into a closed linear system for $(u_h^{n+1}, q_h^{n+1})$ as
\bq\label{nonlin}
\ba
& \left(u_h^{n+1}, \phi \right)  +\frac{\Delta t}{2} \left( b(u_h^n), \phi \right)\left( b(u_h^n), u_h^{n+1} \right)
+\Delta t A(q_h^{n+1}, \phi)
\\
& =\left(u_h^{n}, \phi \right) +\frac{\Delta t}{2} \left(b(u_h^n), \phi \right)\left( b(u_h^n), u_h^{n} \right)-r^n \left(b(u_h^n), \phi \right),\\
& A(u_h^{n+1}, \psi) - (q_h^{n+1},\psi) =  0.
\ea
\eq
This linear system with a nonlocal term $(b(u_h^n), u_h^{n+1})$ has a symmetric yet dense and unstructured  coefficient matrix, and is rather expensive to solve.  
To get around this obstacle, we either return to the augmented system  with $(u_h^{n+1}, q_h^{n+1}, r^{n+1})$
as unknowns, or attempt to find a way to reduce the computational complexity of solving the reduced linear system (\ref{nonlin}).
For the former,  the linear system for the first order scheme is
\begin{subequations}\label{FPDGFull1st+*}
\begin{align}
& (\Delta t)^{-1} \left( u_h^{n+1}, \phi \right) + A(q_h^{n+1},\phi)+r^{n+1} \left( b(u_h^n), \phi \right)=(\Delta t)^{-1}\left( u_h^{n}, \phi \right), \\
& A(u_h^{n+1}, \psi)  - (q_h^{n+1},\psi)=0,\\
&\left(u_h^{n+1},  b(u_h^n)\right) -2 r^{n+1} = \left(u_h^{n},  b(u_h^n)\right)  -2 r^{n}.
\end{align}
\end{subequations}
Though the coefficient matrix of this linear system is also time-dependent, it is sparse and symmetric, hence still suitable for efficient computing. In fact, we use the conjugate gradient (CG) solver to solve this system with the 
 computational complexity of order $O(\mathcal{N})$;  while it is of order $O(\mathcal{N}^2)$ when solving the reduced system (\ref{nonlin}); see, e.g.,  \cite{S94}. 
 
As for the latter, we introduce a special procedure to pre-compute $r^{n+1}$ in order to substantially reduce
the total computational complexity. This treatment is interesting in its own sake.
The details will be presented in the next section.

\section{Pre-evaluation of the auxiliary variable and algorithms}

\subsection{Pre-evaluation of the auxiliary variable $r^{n+1}$}
We introduce an auxiliary linear system: find $(v_h, w_h) \in V_h \times V_h$ such that for $\forall \phi, \psi \in V_h$,
\bq\label{eq2DDGFull}
	\begin{aligned}
	&  \tau A(w_h, \phi)+(v_h, \phi)=(f_h,\phi),\\
	&(w_h, \psi) =  A(v_h,\psi),
	\end{aligned}
\eq
and define operator $(\mathcal{L}_h v, \psi)=A(v, \psi)$ for any $\psi\in V_h$.
We have the following.
\begin{lem} For any $\tau>0$ and $f_h$ given, system (\ref{eq2DDGFull}) admits a unique solution
$(v_h, w_h)$, given by
\bq\label{eq2inv}
v_h=\mathcal{B}_h(\tau)f_h, \quad w_h=\mathcal{L}_hv_h=\mathcal{L}_h \mathcal{B}_h(\tau) f_h.
\eq
Moreover,  the operator $\mathcal{B}_h(\tau)$ can be expressed as $(I+ \tau \mathcal{L}_h^2)^{-1}$,
with the following bounds:
\bq\label{Bff}
( f_h, \mathcal{B}_h(\tau) f_h)=\|\mathcal{B}_h(\tau) f_h\|^2 +\tau\|\mathcal{L}_h \mathcal{B}_h(\tau) f_h\|\geq 0.
\eq
$$
\|\mathcal{B}_h (\tau)f_h\| \leq\|f_h\|.
$$
\end{lem}
\begin{proof}Set $\phi=v_h$ and $\psi=w_h$ in (\ref{eq2DDGFull}) so that
$$
\|v_h\|^2+\tau \|w_h\|^2=(f_h,v_h) \leq \frac12 (\|f_h\|^2 +\|v_h\|^2).
$$
Hence
\bq\label{Bff+}
\|v_h\|^2+ 2 \tau \|w_h\|^2 \leq \|f_h\|^2.
\eq
This a priori estimate ensures both existence and uniqueness of the linear system (\ref{eq2DDGFull}).  Combining  two equations in  (\ref{eq2DDGFull}) we obtain
$$
(\tau \mathcal{L}_h^2+I)v_h=f_h.
$$
This implies that
$$
\mathcal{B}_h(\tau)=(I+\tau \mathcal{L}_h^2)^{-1},
$$
and (\ref{Bff}) follows from (\ref{Bff+}), completing the proof.
\end{proof}
Equipped with the above result, we can compute $r^{n+1}$ in advance for the SAV-DG scheme (\ref{FPDGFull1st+}).
\begin{thm} \label{thm4.1} Givn $(u_h^{n}, q_h^{n})$, scheme (\ref{FPDGFull1st+}) can be realized in
two steps:\\
(i) Determine $r^{n+1}$ by
\bq
\label{rn11st}
r^{n+1} =r^n - \frac{1}{2}(\Pi b(u_h^n), u_h^n)+\frac{1}{2} R^n,
\eq
where
\bq\label{nonloc1st}
\begin{aligned}
R^n = \frac{\left(b(u_h^n), \mathcal{B}_h(\Delta t)\xi^n\right)}
{1+ \frac{\Delta t}{2} \left(\Pi b(u_h^n), \mathcal{B}_h(\Delta t) \Pi b(u_h^n)\right)},
\end{aligned}
\eq
\bq
\label{cn1st}
\xi^n=u_h^n - \Delta t \Pi b(u_h^n) r^n+\frac{\Delta t }{2}\Pi b(u_h^n) \left(b(u_h^n), u_h^n\right);
\eq
(ii) with such obtained $r^{n+1}$ we solve the following linear system:
\bqs
\ba
\left(  D_t u_h^n, \phi \right) = & - A(q_h^{n+1},\phi)-\left( b(u_h^n), \phi \right)r^{n+1}, \\
(q_h^{n+1},\psi) = & A(u_h^{n+1}, \psi).
\ea
\eqs
\end{thm}
\begin{proof} Denote $\mathcal{B}_h=\mathcal{B}_h(\Delta t)$.
From (\ref{FPDGFull1st+}a) we have
\bq\label{FPDGOP}
u_h^{n+1}=u_h^{n}-\Delta t \mathcal{L}_h^2 u_h^{n+1}-\Delta t \Pi b(u_h^n)r^{n+1}\in V_h,
\eq
which further gives
\begin{align*}
 u_h^{n+1} = \mathcal{B}_h u_h^n - \Delta t r^{n+1} \mathcal{B}_h \Pi b(u_h^n).
\end{align*}
Using (\ref{FPDGFull1st+}c), i.e.,
 \bq\label{rnmiddle}
  r^{n+1}=r^n + \frac{1}{2} \left(b(u_h^n),  u_h^{n+1}-u_h^n \right),
\eq
we see that
$$
u_h^{n+1} = \mathcal{B}_h \xi^n - \frac{\Delta t }{2}\mathcal{B}_h\Pi b(u_h^n) \left(b(u_h^n), u_h^{n+1}\right),
$$
where $\xi^n$ is given in (\ref{cn1st}).  Applying inner product against $b(u_h^n)$ gives
\bqs
\ba
 \left(1+  \frac{\Delta t}{2}\left( b(u_h^n), \mathcal{B}_h\Pi b(u_h^n)\right) \right) \left( b(u_h^n), u_h^{n+1}\right) = & \left( b(u_h^n), \mathcal{B}_h \xi^n \right).
\ea
\eqs
Since $(\mathcal{B}_h\Pi b(u_h^n),  b(u_h^n))=(\mathcal{B}_h\Pi b(u_h^n),  \Pi b(u_h^n))\geq 0$, hence,
$$
\left(b(u_h^n), u_h^{n+1}\right) = \frac{ ( b(u_h^n), \mathcal{B}_h \xi^n )}{1+  \frac{\Delta t}{2}\left( \Pi b(u_h^n), \mathcal{B}_h\Pi b(u_h^n)\right)}.
$$
This when inserted into (\ref{rnmiddle}) completes the proof.
\end{proof}

We can also compute $r^{n+1}$ in advance for the second order  SAV-DG scheme
(\ref{FPDGFull+}).
\begin{thm} Given $(u_h^{n}, q_h^{n})$, scheme (\ref{FPDGFull+}) can be realized in
two steps:\\
(i) Determine $r^{n+1/2}$ by
\bq\label{rn12nd}
r^{n+1/2} =r^n -\frac{1}{2}(\Pi b(u_h^{n, *}), u_h^n)+\frac{1}{2}R^{n, *},
\eq
where
\bq\label{nonloc2nd}
R^{n, *}=\frac{\left(b(u_h^{n, *}), \mathcal{B}_h(\Delta t/2)\xi^{n, *}\right)}{1+ \frac{\Delta t}{4}\left(\Pi b(u_h^{n, *}), \mathcal{B}_h(\Delta t/2) \Pi b(u_h^{n,*})\right)},
\eq
\bq\label{cn2nd}
\xi^{n, *}=u_h^n -\frac{1}{2}\Delta t  r^n \Pi b(u^{n,*}_h) +\frac{\Delta t }{4}\Pi b(u^{n,*}_h) \left(b(u^{n,*}_h), u_h^n\right);
\eq
(ii) with such obtained $r^{n+1/2}$ we solve the following linear system:
\bq
\ba
\left(  D_t u_h^n, \phi \right)
	= & - A(q_h^{n+1/2},\phi)-\left(b(u^{n,*}_h),\phi \right)r^{n+1/2},\\
	(q_h^{n}, \psi) = & A(u_h^{n},\psi).\\
\ea
\eq
\end{thm}
\begin{proof} Scheme (\ref{FPDGFull+}) may be rewritten as
	\begin{align*}
	\left(  \tilde D_t u_h^{n+1/2}, \phi \right) = & - A(q_h^{n+1/2},\phi)- r^{n+1/2} \left(b(u^{n,*}_h),\phi \right),\\
	(q_h^{n}, \psi) = & A(u_h^{n},\psi),\\
    \tilde D_t r^{n+1/2} = & \frac{1}{2}\int_\Omega b(u^{n,*}_h) \tilde D_t u_h^{n+1/2}dx,
	\end{align*}
Here $\tilde D_t$ denotes a forward difference with time step $\Delta t/2$ so that
$$
\tilde D_t r^{n+1/2}=D_t r^n.
$$
This is the same form as the first order SAV-DG method with $b(u^{n}_h)$ replaced by $b(u^{n,*}_h)$
and time step $\Delta t$ replaced by $\Delta t/2$. Hence the claimed results follow directly from those in Theorem  \ref{thm4.1}.

\end{proof}

\subsection{Algorithms}
%

The details related to the scheme implementation are summarized in the following algorithms.

\begin{alg}
\label{algorithm1st}  Hybrid algorithm for the first order SAV-DG scheme (\ref{FPDGFull1st+}).
\begin{itemize}
  \item Step 1 (Initialization) From the given initial data $u_0(x)$
  \begin{enumerate}
    \item generate $u_h^0 =\Pi u_0(x) \in V_h $; 
    \item generate $r^0= \sqrt{\int_\Omega\Phi(u_0(x))dx+B}$, where $B$ is a priori chosen constant so that $\inf_v (\int_\Omega\Phi(v(x))dx+B)>0$.
  \end{enumerate}
  \item Step 2 (Evolution)
  \begin{enumerate}
    \item solve for $\Pi b(u_h^n)$ from $b(u_h^n)$;
    \item obtain $\mathcal{B}_h(\Delta t)\Pi b(u_h^n)=v_h$ by solving the linear system (\ref{eq2DDGFull}) with $f_h=\Pi b(u_h^n)$;
    \item obtain $\mathcal{B}_h(\Delta t) \xi^n=v_h$ by solving the linear system (\ref{eq2DDGFull}) with $f_h=\xi_n$ in (\ref{cn1st});
    \item calculate $R^n$ in (\ref{nonloc1st});
    \item calculate $r^{n+1}$ through (\ref{rn11st});
    \item solve the following linear system for $u_h^{n+1}, q_h^{n+1}$,
    \bqs
    \ba
    \left(u_h^{n+1}, \phi \right) + \Delta t A(q_h^{n+1},\phi) = & \left(u_h^{n}, \phi \right) - \Delta t \left( b(u_h^n), \phi \right)r^{n+1}, \\
    A(u_h^{n+1}, \psi)-(q_h^{n+1},\psi) = & 0.
    \ea
    \eqs
  \end{enumerate}
\end{itemize}
\end{alg}

\begin{alg}
\label{algorithm2nd} Hybrid algorithm for the second order SAV-DG scheme (\ref{FPDGFull+}).
\begin{itemize}
  \item Step 1 (Initialization) From the given initial data $u_0(x)$
  \begin{enumerate}
    \item generate $u_h^0 =\Pi u_0(x) \in V_h $; 
    \item solve for $q_h^0$ from (\ref{FPDGFull+}b) based on $u_h^0$;
    \item generate $r^0= \sqrt{\int_\Omega\Phi(u_0(x))dx+B}$, where $B$ is a priori chosen constant so that $\inf_v (\int_\Omega\Phi(v(x))dx+B)>0$.
  \end{enumerate}
  \item Step 2 (Evolution)
  \begin{enumerate}
    \item solve for $\Pi b(u^{n,*}_h)$ based on $b(u^{n,*}_h)$, where $u^{n,*}_h$ is defined in (\ref{u8});
    \item obtain $\mathcal{B}_h(\Delta t/2)\Pi b(u^{n,*}_h)=v_h$ by solving the linear system (\ref{eq2DDGFull}) with $f_h=\Pi b(u^{n,*}_h)$;
    \item obtain $\mathcal{B}_h(\Delta t/2) \xi^{n,*}=v_h$ by solving the linear system (\ref{eq2DDGFull}) with $f_h=\xi^{n,*}$ in (\ref{cn2nd});
    \item calculate $R^{n,*}$ in (\ref{nonloc2nd});
    \item calculate $r^{n+1/2}$ through (\ref{rn12nd});
    \item solve the following linear system for $u_h^{n+1/2}, q_h^{n+1/2}$,
    \bqs
    \ba
    \left(u_h^{n+1/2}, \phi \right)+(\Delta t/2) A(q_h^{n+1/2},\phi) = & \left(u_h^{n}, \phi \right)-
    (\Delta t/2) \left(b(u^{n,*}_h),\phi \right)r^{n+1/2},\\
	A(u_h^{n+1/2}, \psi)-  (q_h^{n+1/2},\psi) = & 0;\\
    \ea
    \eqs
    \item calculate $u_h^{n+1}=2u_h^{n+1/2}-u_h^{n}$.
  \end{enumerate}
\end{itemize}
\end{alg}
Note that each coefficient matrix of the linear system involved in Algorithm \ref{algorithm1st} and \ref{algorithm2nd} is symmetric, sparse and time-independent. The use of the CG  solver  
 for solving these linear systems induces the computational complexity of only order $O(\mathcal{N})$.

\section{Numerical examples}
In this section we numerically test both the spatial and temporal orders of convergence, and apply the second order fully discrete SAV-DG scheme (\ref{FPDGFull+}) to recover roll patterns and hexagonal patterns governed by the two dimensional Swift-Hohenberg equation and further verify the unconditional energy stability of the numerical solutions.

In our numerical tests, we take rectangular meshes. The $L^\infty$ and $L^2$ errors between the numerical solution $u_h^n(x, y)$ and the exact solution $u(t^n,x,y)$  evaluated to obtain experimental orders of convergence (EOC) are defined respectively by
$$
e_h^n = \max_{i}\max_{0\leq l \leq k+1}\max_{0\leq s \leq k+1} |u_h^n(\hat{x}^i_{l},\hat{y}^i_{s})-u(t^n,\hat{x}^i_{l},\hat{y}^i_{s})|
$$
and
$$
e_h^n = \left(\sum_{i} \frac{h^i_xh^i_y}{4} \sum_{l=1}^{k+1}\sum_{s=1}^{k+1} \omega_{l,s} |u_h^n(\hat{x}^i_{l},\hat{y}^i_{s})-u(t^n,\hat{x}^i_{l},\hat{y}^i_{s})|^2\right)^{1/2},
$$
where $\omega_{l,s}>0$ are the weights, and $(\hat{x}^i_{l},\hat{y}^i_{s})$ are the corresponding quadrature points.
The EOC at $T=n\Delta t=2n (\Delta t /2)$ in terms of mesh size $h= \max_i \{ h^i_x, h^i_y\}$ and time step $\Delta t$ are calculated respectively by
$$
\text{EOC}=\log_2 \left( \frac{e_h^n}{e_{h/2}^n}\right),  \quad \text{EOC}=\log_2 \left( \frac{e_h^n}{e_{h}^{2n}}\right).
$$
The Swift-Hohenberg equation is a special case of model equation (\ref{fourthPDE}) with $a=2$  and
$$
 \Psi(u)=\frac{1-\epsilon}{2}u^2 -\frac{g}{3}u^3 +\frac{u^4}{4},
$$
that is,
\begin{align}\label{sh}
u_t  = -\Delta^2 u -2 \Delta u +(\epsilon-1)u +gu^2 -u^3.
\end{align}
Here physical parameters  are $g\geq 0$ and $\epsilon$,  which together with the size of the domain play an important role in pattern selection;  see,  e.g.,  \cite{BPT01, PR04, MD14}.
Our numerical tests center on  this equation for which
$$
\Phi(u)=- \frac{\epsilon}{2}u^2 -\frac{g}{3}u^3 +\frac{u^4}{4}
$$
and $g \geq 0$ and $\epsilon>0$.
This function has double wells with two local minimal values at
$u_\pm=  \frac{g\pm\sqrt{g^2+4\epsilon}}{2}$ such that $\Phi'(u_\pm)=0$, and
\bqs
\Phi(u) \geq \min \{\Phi(u_\pm)\}=
\min_{v=u_\pm} \left( -\frac{1}{12} \left(gv(g^2+4\epsilon)+\epsilon(g^2+3\epsilon) \right) \right)=-a,
\eqs
so it suffices to choose the method parameter $B=a |\Omega|$. In all numerical examples $a<1$, so we simply take $B= |\Omega|$ for all cases.


\begin{example}\label{Ex1dAccS} (Spatial Accuracy Test)
Consider the Swift-Hohenberg equation (\ref{sh}) with an added source of form
$$
f(x,y, t)=- \varepsilon v -gv^2+ v^3, \quad v: =e^{-t/4}\sin(x/2)\sin(y/2),
$$
on $\Omega$, subject to initial data
\bq\label{initex1}
u_0(x,y) = \sin(x/2)\sin(y/2), \quad (x, y) \in \Omega.
\eq
This problem has an explicit solution
\bq\label{uexact}
u(x,y,t) = e^{-t/4}\sin(x/2)\sin(y/2), \quad (x, y) \in \Omega.
\eq
This example is used to test the spatial accuracy, using polynomials of degree $k$ with $k =1,\ 2,\ 3$ on 2D rectangular meshes. 
In the second-order SAV-DG scheme (\ref{FPDGFull+}), we need to add
$$
\frac{1}{2}\left(f(\cdot, t^{n+1}, \phi)+f(\cdot, t^{n}, \phi)\right),
$$
 to the right hand side of (\ref{FPDGFull+}a).  \\

\noindent\textbf{Test case 1.}  We take $\varepsilon=0.025, g=0$,  and domain $\Omega=[-2\pi, 2\pi]^2$ with periodic boundary conditions.  Both errors and orders of convergence at $T=0.01$ are reported in Table \ref{tab2dacc}. These results confirm the $(k+1)$th orders of accuracy in $L^2, L^\infty$ norms.

\begin{table}[!htbp]\tabcolsep0.03in
\caption{$L^2, L^\infty$ errors and EOC at $T = 0.01$ with mesh $N\times N$.}
\begin{tabular}[c]{||c|c|c|c|c|c|c|c|c|c||}
\hline
\multirow{2}{*}{$k$} & \multirow{2}{*}{$\Delta t$}&   \multirow{2}{*}{ } & N=8 & \multicolumn{2}{|c|}{N=16} & \multicolumn{2}{|c|}{N=32} & \multicolumn{2}{|c||}{N=64}  \\
\cline{4-10}
& & & error & error & order & error & order & error & order\\
\hline
\multirow{2}{*}{1}  & \multirow{2}{*}{1e-3} & $\|u-u_h\|_{L^2}$ &  3.18621e-01 & 8.28732e-02 & 1.94 & 2.02935e-02 & 2.03 & 5.04416e-03 & 2.01  \\
\cline{3-10}
 & & $\|u-u_h\|_{L^\infty}$  & 1.38452e-01 & 3.83881e-02 & 1.85 & 9.61389e-03 & 2.00 & 2.40363e-03 & 2.00  \\
\hline
\hline
\multirow{2}{*}{2}  & \multirow{2}{*}{1e-4} & $\|u-u_h\|_{L^2}$ & 6.96867e-02 & 1.49828e-02 & 2.22 & 2.01641e-03 & 2.89 & 2.56761e-04 & 2.97  \\
\cline{3-10}
 & & $\|u-u_h\|_{L^\infty}$  & 2.41046e-02 & 2.94730e-03 & 3.03 & 4.02470e-04 & 2.87 & 5.14111e-05 & 2.97  \\
 \hline
\hline
\multirow{2}{*}{3} & \multirow{2}{*}{1e-5}  & $\|u-u_h\|_{L^2}$ & 1.19940e-02 & 1.13110e-03 & 3.41 & 7.72013e-05 & 3.87 & 5.01113e-06 & 3.95  \\
\cline{3-10}
 & & $\|u-u_h\|_{L^\infty}$   & 3.85634e-03 & 3.68735e-04 & 3.39 & 2.43503e-05 & 3.92 & 1.53912e-06 & 3.98  \\
\hline
\end{tabular}\label{tab2dacc}
\end{table}

\noindent\textbf{Test case 2.} We take $\varepsilon=0.025, g=0.05$,  domain $\Omega=[-\pi, 3\pi]^2$ with   boundary condition $\partial_\nu u = \partial_\nu \Delta u = 0, \ (x,y) \in \partial \Omega$.  Both errors and orders of convergence at $T=0.01$ are reported in Table \ref{tab2daccNeu}. These results also show that $(k+1)$th orders of accuracy in both $L^2$ and $L^\infty$ norms are obtained.

\begin{table}[!htbp]\tabcolsep0.03in
\caption{$L^2, L^\infty$ errors and EOC at $T = 0.01$ with mesh $N\times N$.}
\begin{tabular}[c]{||c|c|c|c|c|c|c|c|c|c||}
\hline
\multirow{2}{*}{$k$} & \multirow{2}{*}{$\Delta t$}&   \multirow{2}{*}{ } & N=8 & \multicolumn{2}{|c|}{N=16} & \multicolumn{2}{|c|}{N=32} & \multicolumn{2}{|c||}{N=64}  \\
\cline{4-10}
& & & error & error & order & error & order & error & order\\
\hline
\multirow{2}{*}{1}  & \multirow{2}{*}{1e-3} & $\|u-u_h\|_{L^2}$ &  3.18621e-01 & 8.28732e-02 & 1.94 & 2.02935e-02 & 2.03 & 5.04416e-03 & 2.01  \\
\cline{3-10}
 & & $\|u-u_h\|_{L^\infty}$  & 1.38452e-01 & 3.83886e-02 & 1.85 & 9.61391e-03 & 2.00 & 2.40363e-03 & 2.00  \\
\hline
\hline
\multirow{2}{*}{2}  & \multirow{2}{*}{1e-4} & $\|u-u_h\|_{L^2}$ & 6.96867e-02 & 1.49828e-02 & 2.22 & 2.01641e-03 & 2.89 & 2.56762e-04 & 2.97  \\
\cline{3-10}
 & & $\|u-u_h\|_{L^\infty}$  & 2.41054e-02 & 2.94731e-03 & 3.03 & 4.02470e-04 & 2.87 & 5.14110e-05 & 2.97  \\
 \hline
\hline
\multirow{2}{*}{3} & \multirow{2}{*}{1e-5}  & $\|u-u_h\|_{L^2}$ & 1.19940e-02 & 1.13110e-03 & 3.41 & 7.72042e-05 & 3.87 & 5.05657e-06 & 3.93  \\
\cline{3-10}
 & & $\|u-u_h\|_{L^\infty}$   & 3.85659e-03 & 3.68738e-04 & 3.39 & 2.43504e-05 & 3.92 & 1.53917e-06 & 3.98  \\
\hline
\end{tabular}\label{tab2daccNeu}
\end{table}

\end{example}

\begin{example}\label{Ex1dAccT} (Temporal Accuracy Test)  Consider the Swift-Hohenberg equation with source term given as in Example \ref{Ex1dAccS}.  We take $\varepsilon=0.025$ and $g=0$, and domain $\Omega= [-4\pi, 4\pi]^2$ with periodic boundary conditions, subject to initial data
\bq\label{initex12}
u_0(x,y) = \sin(x/4)\sin(y/4).
\eq
Its exact solution is given by
$$
u(x,y,t) = e^{-49t/64}\sin(x/4)\sin(y/4), \quad (x, y) \in \Omega.
$$
We compute the numerical solutions using the SAV-DG schemes (\ref{FPDGFull1st+}) and (\ref{FPDGFull+}) based on $P^2$ polynomials with time steps $\Delta t=2^{-m}$ for $2\leq m\leq 5$ and mesh size $64\times 64$. The $L^2, L^\infty$ errors and orders of convergence at $T=2$ are shown in Table \ref{timeacc}, and these results confirm that DG schemes (\ref{FPDGFull1st+}) and (\ref{FPDGFull+}) are first order and second order in time, respectively.

\begin{table}[!htbp]\tabcolsep0.03in
\caption{$L^2, L^\infty$ errors and EOC at $T = 2$ with time step $\Delta t$.}
\begin{tabular}[c]{||c|c|c|c|c|c|c|c|c|c||}
\hline
\multirow{2}{*}{Scheme} & \multirow{2}{*}{Mesh}&   \multirow{2}{*}{ } & $\Delta t=2^{-2}$ & \multicolumn{2}{|c|}{$\Delta t=2^{-3}$} & \multicolumn{2}{|c|}{$\Delta t=2^{-4}$} & \multicolumn{2}{|c||}{$\Delta t=2^{-5}$}  \\
\cline{4-10}
& & & error & error & order & error & order & error & order\\
\hline
\multirow{2}{*}{(\ref{FPDGFull1st+})}  & \multirow{2}{*}{$64\times64$} & $\|u-u_h\|_{L^2}$ &  3.05892e-01 & 1.58442e-01 & 0.95 & 8.07023e-02 & 0.97 & 4.07442e-02 & 0.99  \\
\cline{3-10}
 & & $\|u-u_h\|_{L^\infty}$  & 1.75153e-02 & 9.09087e-03 & 0.95 & 4.64080e-03 & 0.97 & 2.34881e-03 & 0.98  \\
\hline
\hline
\multirow{2}{*}{(\ref{FPDGFull+})}  & \multirow{2}{*}{$64\times64$} & $\|u-u_h\|_{L^2}$ & 4.17744e-02 & 8.14437e-03 & 2.36 & 1.74312e-03 & 2.22 & 3.98404e-04 & 2.13  \\
\cline{3-10}
 & & $\|u-u_h\|_{L^\infty}$  & 4.01428e-03 & 7.92985e-04 & 2.34 & 1.46602e-04 & 2.44 & 3.60847e-05 & 2.02  \\
 \hline
\end{tabular}\label{timeacc}
\end{table}

\end{example}

\begin{example} We consider the Swift-Hehenberg equation with the parameters in Example \ref{Ex1dAccT}. Here we compare the computational complexity of (\ref{nonlin}), (\ref{FPDGFull1st+*}) and Algorithm \ref{algorithm1st} in implementing the first order SAV-DG scheme (\ref{FPDGFull1st+}). We use $P^1$ polynomials with time step $\Delta t=10^{-2}$ and meshes $N\times N$. The total CPU time  and the orders of the CPU time relative to the number of unknowns  are presented in Table \ref{tab2daccNeu32}.

Let $\mathcal{N}=6N^2+1$  be the total number of unknowns. The results tell us that the computational complexity of (\ref{nonlin}) is $O(\mathcal{N}^2)$, but only $O(\mathcal{N})$ for (\ref{FPDGFull1st+*}) and Algorithm \ref{algorithm1st}.
The key for the $O(\mathcal{N})$ complexity lies in the sparsity of the coefficient matrix, however, (\ref{FPDGFull1st+*})
solves a larger system, and Algorithm \ref{algorithm1st} involves a pre-evaluation procedure. Still Algorithm \ref{algorithm1st} appears best among all three methods.

\begin{table}[!htbp]\tabcolsep0.03in
\caption{The CPU time in seconds with respect to meshes $N\times N$ at $T=0.1$.}
\begin{tabular}[c]{||c|c|c|c|c|c|c|c||}
\hline
\multirow{2}{*}{method} & N=8 & \multicolumn{2}{|c|}{N=16} & \multicolumn{2}{|c|}{N=32} & \multicolumn{2}{|c||}{N=64}  \\
\cline{2-8}
& CPU time & CPU time & order & CPU time & order & CPU time & order \\
\hline
(\ref{nonlin}) & 3.46 & 16.58 & 1.13 & 157.58 & 1.63 & 2687.80 & 2.05\\
\hline
(\ref{FPDGFull1st+*}) & 2.70 & 9.79 & 0.93 & 38.10 & 0.98 & 155.07 & 1.01 \\
\hline
Algorithm \ref{algorithm1st} & 2.10 & 7.39 & 0.91 & 28.49 & 0.97 & 116.00 & 1.01 \\
\hline
\end{tabular}\label{tab2daccNeu32}
\end{table}
\end{example}

\section{Concluding remarks} For a class of fourth order gradient flows, integration of the spatial discretization based on the penalty-free DG method  introduced in \cite{LY18} with the temporal discretization based on the SAV approach introduced in \cite{SXY18} to handle nonlinear terms led us to SAV-DG schemes. Such schemes inherit the energy dissipation property of the continuous equation irrespectively of the mesh and time step sizes.
However, the resulting linear system involving unknowns $(u, q)$ only, where  $q$ is an approximation of $\mathcal{L}=-\left(\Delta+\frac{a}{2}\right)u$, is rather expensive to solve due to the dense coefficient matrix.
In this paper, we have developed hybrid SAV-DG algorithms in two steps: we (i) provide a procedure to pre-evaluate the auxiliary variable $r^{n+1}$ in the piecewise polynomial space, and (ii) solve the resulting linear system with the obtained $r^{n+1}$. This procedure reduced the computational complexity of the CG solver to $O(\mathcal{N})$ from $O(\mathcal{N}^2)$; here $\mathcal{N}$ is the total number of unknowns.
We also presented several numerical examples to assess the performance of the hybrid SAV-DG algorithms in terms of accuracy and energy stability. Also the cost of the hybrid SAV-DG is comparable to that for solving the augmented system involving $(u, q, r)$, with the hybrid SAV-DG performing better as evidenced by our numerical results.

\section*{Acknowledgments}
 This research was supported by the National Science Foundation under Grant DMS1812666.

\bigskip
\appendix

\section{Proofs of energy dissipation laws}.
\begin{proof}
(i) We first prove (\ref{engdis1st+}). From (\ref{FPDGFull1st+}b), it follows
\begin{align}
(D_tq_h^n, \psi)=A(D_tu_h^n, \psi).
\end{align}
Taking $\psi=q_h^{n+1}$ and $\phi=D_t u_h^n$ in (\ref{FPDGFull1st+}a), when combined with (\ref{FPDGFull1st+}c) we have
\bq\label{exist1st}
\ba
-\|D_tu_h^n\|^2  =& (D_t q_h^{n}, q_h^{n+1}) +(b(u_h^n), D_tu_h^n)r^{n+1}\\
 =& \frac{1}{2} D_t \|q_h^n\|^2 +\frac{\Delta t }{2} \|D_t q_h^n\|^2 +2 r^{n+1}D_t r^n\\
 = &\frac{1}{2} D_t \|q_h^n\|^2 +\frac{\Delta t }{2} \|D_t q_h^n\|^2 +  D_t |r^n|^2 + \Delta t |D_t r^n|^2,
\ea
\eq
which leads to the desired equality (\ref{engdis1st+}).

Next we show the uniqueness of the SAV-DG scheme (\ref{FPDGFull1st+}). Let  $(\tilde u, \tilde q, \tilde r)$ be the difference of two possible solutions at $t=t_{n+1}$, then (\ref{exist1st}) is equivalent to
$$
\frac{1}{\Delta t}\|\tilde u\|^2+\|\tilde q\|^2+2|\tilde r|^2=0,
$$
hence we must have $(\tilde u, \tilde q, \tilde r)=(0, 0, 0)$,
leading to the uniqueness of the linear system (\ref{FPDGFull1st+}), hence its existence  since for a linear system in finite dimensional space, existence is equivalent to its uniqueness.

(ii) We first prove (\ref{engdis}). From (\ref{FPDGFull+}b), it  follows
\bq
(D_t q_h^n, \psi) = A(D_t u_h^n, \psi).
\eq
Taking $\psi=q_h^{n+1/2}$ and $\phi=D_t u_h^n$ in (\ref{FPDGFull+}a), when combined with (\ref{FPDGFull+}c) we have
\begin{align*}
-\|D_tu_h^n\|^2 & =  (D_t q_h^{n}, q_h^{n+1/2}) + (b(u_h^{n,*})r^{n+1/2}, D_tu_h^n) =\frac{1}{2} D_t \|q_h^n\|^2 + D_t |r^n|^2.
\end{align*}
Multiplying by $\Delta t$ on both sides of this equality leads to (\ref{engdis}).

Similar to (i), the existence of the SAV-DG scheme (\ref{FPDGFull+}) is equivalent to its uniqueness,
we let  $(\tilde u, \tilde q, \tilde r)$ be the difference of two possible solutions at $t=t_{n+1}$ again, then a similar analysis yields
$$
\frac{1}{\Delta t}\|\tilde u\|^2+\frac{1}{2}\|\tilde q\|^2+|\tilde r|^2=0,
$$
hence we must also have $(\tilde u, \tilde q, \tilde r)=(0, 0, 0)$, leading to the uniqueness of the scheme (\ref{FPDGFull+}).
\end{proof}

\end{document}